\documentclass[11pt, oneside]{article}   	
\usepackage{geometry}                		
\usepackage{graphicx}				
\usepackage{amssymb}
\usepackage{amsmath}
\usepackage{amsthm}
\usepackage{harpoon}
\usepackage{accents}

\usepackage{tikz}  
\usepackage{float}
\restylefloat{table}
\usepackage{mathrsfs}
\usepackage{verbatim}
\usepackage{enumitem}  

\usetikzlibrary{positioning,chains,fit,shapes,calc}  

\newtheorem{theorem}{Theorem}
\newtheorem{lemma}{Lemma}
\newtheorem{definition}{Definition}

\newtheorem{remark}{Remark}
\newtheorem{corollary}{Corollary}

\newcommand{\ppp}{\[\begin{aligned}}
\newcommand{\ooo}{\end{aligned}\]}
\newcommand{\err}{\mathrm{err}}
\newcommand{\rad}{\mathrm{rad}}
\newcommand{\val}{\mathrm{val}}
\newcommand{\ind}{\mathrm{ind}}
\newcommand{\mrm}[1]{\mathrm{#1}}

\begin{document}
\title{On recurrence relations arising from NRS(2) applied to a cubic polynomial}
\author{Mario DeFranco}

\maketitle

\abstract{We prove that the leading coefficient of the ``error" terms of NRS(2) applied to a cubic polynomial $f(z)$ with starting point $(-\frac{a_1}{a_2}, -\frac{a_1}{a_2})$ are positive-coefficient rational functions in the zeros of $f(z)$. We express these terms as a sum over combinatorial objects which we call radius-value trees.}

\section{Introduction} 

Let $f(z) \in \mathbb{C}[z]$ be a degree $d$ polynomial 
\[
f(z) = \sum_{i=0}^d a_i z^i = \prod_{i=1}^d (1-u_iz)
\]
where we assume $a_0 = 1$. We consider the iterations of NRS(2) applied to $f(z)$ for $d=3$ in terms of the reciprocal zeros $u_i$. We express the difference between the $n$-th iteration and the auxiliary-function zero
\[
(\alpha_0, \alpha_1)
\]
(see \cite{DeFranco 1} and \cite{DeFranco 2}) using polynomials in $u_1$ and $u_2$. In particular, we focus on the leading coefficient of $u_3$ in these polynomials, which is a polynomial in $u_1$ and $u_2$. We obtain recurrence relations for these leading coefficients in section \ref{rr}. 

Our main results are in section \ref{pos}. Theorem \ref{t E RV} proves that these leading coefficients can be expressed as a sum over objects we call radius-value trees, defined in section \ref{rv trees}. Theorem \ref{t partition} and Corollary \ref{c partition} prove that this sum has all positive terms. We express these results using the ring $\mrm{CH}_2$ of complete homogeneous polynomials in $u_1$ and $u_2$, and also the ring $\tilde{\mrm{CH}}_2$, defined in section \ref{ch}. 

\section{NRS(2) applied to cubic polynomials} \label{rr}
Let $f(z) \in \mathbb{C}[z]$ be a degree $d$ polynomial 
\[
f(z) = \sum_{i=0}^d a_i z^i = \prod_{i=1}^d (1-u_iz)
\]
where we assume $a_0 = 1$.
Recall that the auxiliary functions for NRS(2) applied to $f(z)$ are

\begin{align*}
f_0(x_0,x_1) &= -\sum_{i=-1}^{d-1}\sum_{j=0}^{\lfloor \frac{i}{2} \rfloor}  -\frac{a_{i+1}}{a_2}((\frac{-a_0}{a_2})(\frac{-a_1}{a_2})^{-1} x_1)^{j} x_0^{i-2j}{i-j \choose j}\\ 
f_1(x_0, x_1) &= -\sum_{i=-2}^{d-2}\sum_{j=-1}^{\lfloor \frac{i}{2} \rfloor}  -\frac{a_{i+2}}{a_2}((\frac{-a_0}{a_2})(\frac{-a_1}{a_2})^{-1} x_1)^{j} x_0^{i-2j}x_1{i-j \choose j}.
\end{align*}
When $d=3$, these become 
\begin{align*}
f_0(x_0, x_1) &= x_0+\frac{a_1}{a_2}+ x_0^2\frac{a_3}{a_2} + x_1\frac{a_0 a_3}{a_1a_2} \\ 
f_1(x_0, x_1) &=x_1 + \frac{a_1}{a_2} + x_0 x_1 \frac{a_3}{a_2}.
\end{align*}

Let 
\[
(c_{n,0}, c_{n,1})
\]
be the output of the $n$-th iteration of NRS(2), i.e. 
\begin{equation} \label{NRS2 rr}
(c_{n+1,0}, c_{n+1,1}) = (c_{n,0}, c_{n,1}) - J_n^{-1}(f_0(c_{n,0}, c_{n,1}), f_1(c_{n,0}, c_{n,1}))
\end{equation}
where $J_n$ denotes the Jacobian matrix of $f_0$ and $f_1$ evaluated at 
\[
(x_0, x_1) = (c_{n,0}, c_{n,1}).
\]
Define polynomials $\err(n,0), \err(n,1), \err(n,-1)$ in $u_1$ and $u_2$ by
\begin{align}
\frac{\err(n,0)}{\err(n,-1)} &= \frac{u_1+u_2}{u_1u_2} - c_{n,0} \label{errfrac 0}\\ 
\frac{\err(n,1)}{\err(n,-1)} &= \frac{u_1+u_2+u_3}{u_1u_2} - c_{n,1}.\label{errfrac 1} 
\end{align}

Then we may express the recurrence relations \eqref{NRS2 rr} for NRS(2) as 
\begin{equation} \label{orr}
\underline{\text{Original recurrence relations}}
\end{equation}

\begin{align*} 
\err(0,0) &= (u_1^2 + u_1u_2+ u_2^2)u_3\\ 
\err(0,1) &= (u_1 + u_2)u_3(u_1+u_2+u_3) \\ 
\err(0,-1) &= u_1u_2(u_1u_2+u_1u_3+u_2u_3)
\end{align*} 

\begin{align*} 
\err(n+1,0) &=-u_1 u_2 u_3^2 \err(n,0)\err(n,-1) \err(n,1)+ 
  u_1 u_2 u_3 (u_1+u_2+u_3)\err(n,0)^2\err(n,-1) \\ 
  &\,\,\,\,\,\,+   u_1 u_2 u_3^2 (u_1+u_2+u_3)\err(n, 0)^3\\ 
\err(n+1,1) &=u_3 (u_1+u_2+u_3) (u_1 u_2 - u_1 u_3 - u_2 u_3) \err(n,0) \err(n,1) \err(n,-1) \\ 
&\,\,\,\,\,\,+  
 u_3^2(u_1+u_2+u_3)^2 \err(n,0)^2 \err(n,-1) \\ 
 & \,\,\,\,\,\,+ 
 u_1 u_2 u_3^2 (u_1+u_2+u_3) \err(n,0)^2 \err(n,1) \\ 
\err(n+1,-1) &= \err(n,-1) ((u_1 - u_3) (u_2 - u_3) (u_1+u_2+u_3) \err(n,-1)^2 - u_1 u_2 u_3^2 \err(n,1) \err(n,-1)\\ 
&  \,\,\,\,\,\, + 
   u_3(u_1+u_2+u_3) (3 u_1 u_2 - u_1 u_3 - u_2 u_3) \err(n,0) \err(n,-1) \\ 
   & \,\,\,\,\,\,+ 2 u_1 u_2 u_3^2 (u_1+u_2+u_3)\err(n,0)^2
\end{align*} 

We apply some changes of variables in order to simplify these relations. First, for each $n$
\[
\err(n,1) \mapsto \err(n,1) (u_1+u_2+u_3).
\]
After this change, we may factor out $(u_1+u_2+u_3)$ from $\err(n+1,0), \err(n+1,1)$ and $\err(n,-1)$ which cancel after making the fractions \eqref{errfrac 0} and \eqref{errfrac 1}. Then make
\[
u_3 \mapsto u_1u_2u_3
\]
and then 
\begin{align*}
\err(n,0) &\mapsto (u_1u_2)^{\frac{5(3^n)-3}{2}}\err(n,0) \\ 
\err(n,1) &\mapsto (u_1u_2)^{\frac{5(3^n)-3}{2}}\err(n,1)\\ 
\err(n,-1) &\mapsto (u_1u_2)^{\frac{5(3^n)-1}{2}}\err(n,-1)  
\end{align*}
and then 
\begin{align*}
\err(n,0) &\mapsto u_3^{2^{n+1}-1}\err(n,0) \\ 
\err(n,1) &\mapsto u_3^{2^{n+1}-1}\err(n,1).
\end{align*}

These changes yield the
\begin{equation} \label{mrr}
\underline{\text{Modified recurrence relations}}
\end{equation}

\begin{align*} 
\err'(0,0) &= u_1^2 + u_1u_2+ u_2^2\\ 
\err'(0,1) &= u_1 + u_2\\ 
\err'(0,-1) &= 1+ u_3(u_1+u_2)
\end{align*} 

\begin{align*} 
\err'(n+1,0) &=\err'(n,0) (\err'(n,-1)\err'(n,0) - \err'(n,-1)\err'(n,1) u_1 u_2 u_3 +
    \err'(n,0)^2 u_3^{2^{1 + n}})\\ 
\err'(n+1,1) &=\err'(n,0) (\err'(n,-1)\err'(n,1) + \err'(n,-1)\err'(n,0) u_3 - 
   \err'(n,-1)\err'(n,1) (u_1 +u_2)u_3 \\ 
   & \,\,\,\,\,\,+ 
   \err'(n,0) \err'(n,1) u_3^{2^{1 + n}})\\ 
\err'(n+1,-1) &= \err'(n,-1)(\err'(n,-1)^2 - \err'(n,-1)^2 (u_1+u_2) u_3 
 + \err'(n,-1)^2 u_1 u_2 u_3^2 \\ 
   &\,\,\,\,\,\,+ 
   3 \err'(n,-1)\err'(n,0) u_3^{2^{1 + n}} + 
   2 \err'(n,0)^2 u_3^{2^{2 + n}} - 
   \err'(n,-1)\err'(n,0) (u_1 +u_2)u_3^{1 + 2^{1 + n}} \\ 
   &\,\,\,\,\,\, - \err'(n,-1)\err'(n,1) u_1 u_2 u_3^{1 + 2^{1 + n}})\\
\end{align*} 

 \begin{lemma} \label{l u3 deg} 
As polynomials in $u_3$, assuming $u_1$ and $u_2$ to be indeterminates, we have the degrees 
\begin{align*}
\deg_{u_3}(\err'(n,-1) )&= 2(3^n)-1\\ 
\deg_{u_3}(\err'(n,0) )&= 2(3^n - 2^n)\\
\deg_{u_3}(\err'(n,1) )&= 2(3^n - 2^n) 
\end{align*}
 \end{lemma}
\begin{proof}
Straightforward using induction on $n$. This completes the proof.  
\end{proof}

Denote by $\err_0(n,i)$ the leading coefficient of $\err'(n,i)$ with respect to $u_3$, for $i=-1,0,1$. We then obtain 
\begin{equation} \label{mlcr}
\underline{\text{Modified leading coefficient relations}}
\end{equation}
\begin{align*} 
\err_0(0,0) &= u_1^2 + u_1u_2+ u_2^2\\ 
\err_0(0,1) &= u_1 + u_2\\ 
\err_0(0,-1) &= u_1+u_2
\end{align*}
 
 \begin{align*} 
\err_0(n+1,0) &=(\err_0(n,0)^2 -u_1u_2\err_0(n,-1) \err_0(n,1)) \err_0(n,0)\\ 
\err_0(n+1,1) &=\err_0(n,0)^2(\err_0(n,-1)+ \err_0(n,1) ) - 
   (u_1+u_2)\err_0(n,-1)\err_0(n,1) \err_0(n,0)\\ 
\err_0(n+1,-1) &= \err_0(n,-1)(u_1u_2(\err_0(n,-1)^2 - \err_0(n,-1)\err_0(n,1))\\
&\,\,\,\,\,\,+ 2 \err_0(n,0)^2 - (u_1+u_2)\err_0(n,0)\err_0(n,-1))
\end{align*}
 
\begin{lemma} \label{l mlcre}
For all $n$ 
\[
\err_0(n,-1) = \err_0(n,1)
\]
and thus the modified leading coefficient relations \eqref{mlcr} are equivalent to 
\begin{align*} 
\err_0(0,0) &= u_1^2 + u_1u_2+ u_2^2\\ 
\err_0(0,1) &= u_1 + u_2\\ 
\end{align*}
\begin{align*} 
\err_0(n+1,0) &=(\err_0(n,0)^2 -u_1u_2 \err_0(n,1)^2) \err_0(n,0)\\ 
\err_0(n+1,1) &=(2\err_0(n,0)^2 - 
   (u_1+u_2)\err_0(n,1) \err_0(n,0)) \err_0(n,1).\\ 
\end{align*}
\end{lemma} 
 \begin{proof}
 Straightforward induction on $n$. Setting $\err_0(n,-1) = \err_0(n,1)$ reduces the relations \eqref{mlcr} to those of the lemma. This completes the proof.  
 \end{proof}

\section{The rings $\mathrm{CH}_2$ and $\tilde{\mathrm{CH}}_2$} \label{ch}
Let $\mathrm{CH}_2$ denote the ring generated by complete homogeneous polynomials $h_k$ in two indeterminates $u_1$ and $u_2$, where 
\[
h_k = \sum_{i=0}^k u_1^{k-i} u_2^i.
\]

We describe how to multiply $h_i$ and $h_j$ in terms of centered paths defined next.  
\begin{definition} 
For an integer $r \geq 0$, let $P$ be path on the set $V(P)$ of $r+1$ vertices indexed as
\[
V(P) = (P(-r), P(-r+2), \ldots, P(r-2), P(r))
\] 
equipped with directed edges 
\[
P(-r+2i) \rightarrow P(-r+2i+2)
\]
for $0\leq  i < r$. We call $P$ a \emph{centered path} and call $r$ the \emph{radius of $P$} denoted by $\mathrm{rad}(P)$. If $r$ is even, then call $P(0)$ the central vertex of the path. 

Suppose $P_1$ and $P_2$ are centered paths with $r_i = \mathrm{rad}(P_i)$ for $i=1,2$. For $0 \leq k \leq \min(r_1,r_2)$, define the $k$-th product path 
\[
\pi(P_1,P_2, k)
\]
whose vertex set is a subset of $V(P_1) \times V(P_2)$ with edges defined as follows. For $0 \leq j \leq r_2 - k -1$
\[
(P_1(-r_1+ 2k), P_2(-r_2+2j)) \rightarrow (P_1(-r_1+ 2k), P_2(-r_2+2j+2)), 
\] 
and for $1 \leq j \leq r_1 - k -1$
\[
( P_1(-r_1+ 2k+2j), P_2(r_2-2k)) \rightarrow (P_1(-r_1+ 2k+2j+2), P_2(r_2-2k)). 
\] 

$\pi(P_1,P_2,k)$ 
For integers $r_1, r_2 \geq 0$ and $x,y$ define 
\[
\rho(r_1, r_2, x, y) = \begin{cases} 
r_1+y \text{ if } y \geq -x+r_2-r_1 \\ 
r_2 - x \text{ if }  y< -x+r_2-r_1
\end{cases}. 
\]
\end{definition}

\begin{lemma} \label{l cp facts}

i).
\[
\mathrm{rad}(\pi(P_1,P_2,k)) = \mathrm{rad}(P_1)+\mathrm{rad}(P_2) - 2k
\]

ii). The radius of the centered path that contains $(P_1(x), P_2(y))$ is 
\[
\rho(\mathrm{rad}(P_1), \mathrm{rad}(P_2), x, y).
\]

iii). We have
\[
\bigcup_{k=0}^{\min(r_1,r_2)}  V(\pi(P_1,P_2, k)) = V(P_1) \times V(P_2)
\]
and the above union is disjoint.

iv). Suppose 
\[
(P_1(x), P_2(y)) = \pi(P_1, P_2 ,k)(j).
\]
Then 
\[
j = x+y.
\]
\end{lemma}
\begin{proof} 
These statements follow immediately by associating $(P_1(x), P_2(y)) \in V(P_1) \times V(P_2)$ with the point $(x,y)$ in the $x-y$ plane and considering the line 
\[
y = -x-r_1+r_2.
\]
For each point on this line 
\[
(-r_1+k, r_2-k)
\]  
the set consisting of this point, the points on the vertical ray below it and all the points on the horizontal ray to the right constitute the vertex set of $\pi(P_1,P_2, k)$. This completes the proof. 
 \end{proof}

Lemma \ref{l cp facts} iii implies that 
\[
h_i h_j = \sum_{k=0}^{(j+i -|j-i|)/2}(u_1 u_2)^{(j+i - |j-i|-2k)/2}h_{|j-i|+2k}. 
\]
Since the power of $u_1 u_2$ is determined by $i,j$ and $k$, we may in the following suppress the notational dependence on these factors of $u_1u_2$, remembering that $h_i$ in an expression of total $(u_1,u_2)$ degree $t$ now means 
\[
(u_1 u_2)^{(t - i)/2} \sum_{k=0}^i u_1^k u_2^{i-k}
\]
where $t \equiv i \mod 2$ always. Thus we define the $\mrm{CH}_2$ as follows. 

\begin{definition} \label{d M rings}

Let $\mrm{CH}_2$ denote the ring generated by the elements $\{ h_i \colon i \in \mathbb{Z}, i \geq 0\}$ with the relations 
\[
h_i h_j = \sum_{k=0}^{(j+i -|j-i|)/2}h_{|j-i|+2k}. 
\]
\end{definition}

\begin{definition} 
Define 
\begin{align*} 
\tau \colon \mathbb{Z} &\rightarrow \mathbb{Z} \\ 
i &\mapsto |i+1|-1.
\end{align*}

Let $\tilde{\mrm{CH}}_2$ denote the ring generated by the elements $\{ \tilde{h}_i \colon i \in \mathbb{Z}\}$ with the relations 
\[
\tilde{h}_i \tilde{h}_j = 
\begin{cases} 
0 &\text{ if } i = -1 \\
\sum_{k=0}^i \tilde{h}_{j-i+2k} &\text{ if } i \geq 0\\
-\tilde{h}_{\tau(i)}\tilde{h}_j &\text{ if } i <-1
\end{cases}
\]
\end{definition} 

\begin{remark} 
These multiplication rules are associative as seen by checking them on generators. Thus the rings $\mrm{CH}_2$ and $\tilde{\mrm{CH}}_2$ are well-defined. 
\end{remark}

\begin{definition} 
Let $S_k$ denote the linear operator 

\begin{align*}
S_k\colon \tilde{\mrm{CH}}_2 &\rightarrow \tilde{\mrm{CH}}_2 \\ 
\tilde{h_i} &\mapsto \tilde{h}_{i+k}
\end{align*}

Let $L$ denote the ring homomorphism 
\begin{align*}
L \colon \tilde{\mrm{CH}}_2 &\rightarrow \mrm{CH}_2 \\
\tilde{h}_i &\mapsto \mathrm{sgn}(i+1)h_{\tau(i)}
\end{align*}
\end{definition} 

\begin{remark}
That $L$ is a homomorphism is straightforwardly checked by its action on generators. The kernel $\mathrm{Ker}(L)$ is generated by $\tilde{h}_{-1}$ and the elements $\tilde{h}_{-i}+ \tilde{h}_{i-2}$. 
\end{remark}

\begin{lemma} \label{l T g action}
For $g_1, g_2 \in \tilde{\mrm{CH}}_2$, suppose $g_1 \in \mathrm{Ker}(L)$. Then 
\[
g_1g_2 =0.
\]  
\end{lemma} 
\begin{proof}
It is sufficient to check the lemma when $g_1$ is a generator of $\mathrm{Ker}(L)$ and $g_2$ is a generator of $\tilde{\mrm{CH}}_2$. Then it follows from straightforward application of Definition \ref{d M rings}. This completes the proof.  
\end{proof}

\begin{lemma} \label{l oe cont}
\begin{align}
\tilde{h}_{a}\tilde{h}_b - \tilde{h}_{a-1}\tilde{h}_{b-1} &= \tilde{h}_{a+b} \label{e cont}\\ 
\tilde{h}_{a}\tilde{h}_b - \tilde{h}_{a-2}\tilde{h}_{b} &= \tilde{h}_{b+a} + \tilde{h}_{b-a} \label{o cont}
\end{align}
\end{lemma} 
\begin{proof}
Straightforward application of definitions depending on the cases $a\geq 0$, $a=-1$, or $a <-1$.  
\end{proof}

\begin{definition} 
Denote
\begin{align*}
U(a,b,c) &= \mathrm{sgn}(a)\sum_{j=0}^{|a|-1}\tilde{h}_{c-b-|a|+2j} \\ 
\end{align*}
with $U(0,b,c)=0$.
\end{definition} 

\begin{lemma} \label{l U}
\begin{align}
\tilde{h}_{a+b}\tilde{h}_{c}  - \tilde{h}_{b}\tilde{h}_{a+c}&= U(a,b,c) \label{e U} \\
(\tilde{h}_{b+a} + \tilde{h}_{b-a}) \tilde{h}_{c-1} - \tilde{h}_{b}\tilde{h}_{a+c-1} &=U(b+1,a,c) \label{o U}
\end{align}
\end{lemma} 
\begin{proof}
We first prove \eqref{e U} in the case $b=-1$. Then if $a > 0$
\[
\tilde{h}_{a-1}\tilde{h}_{c}  = \sum_{j=0}^{a-1}\tilde{h}_{c-a+1+2j}.
\]
If $a=0$, then 
\[
\tilde{h}_{-1}\tilde{h}_{c} =0.
\]
If $a <0$, then 
\begin{align*}
\tilde{h}_{a-1}\tilde{h}_{c} &= -\tilde{h}_{-a-1}\tilde{h}_{c} \\
&= -\sum_{j=0}^{-a-1}\tilde{h}_{c+a+1+2j}.
\end{align*}
This proves the case $b=-1$ for \eqref{e U}. For \eqref{o U} with $b=-1$, the left side is 
\[
(\tilde{h}_{-1+a} + \tilde{h}_{-1-a}) \tilde{h}_{c-1}
\]
which is equal to 0 by Lemma \ref{l T g action} because 
\[
\tilde{h}_{-1+a} + \tilde{h}_{-1-a} \in \mathrm{Ker}(L).
\]

Now assume $b \neq -1$. To prove \eqref{e U} we consider the six cases
\begin{align*} 
1.1)& a+b \geq 0, b \geq 0, a\geq 0\\
1.2)& a+b \geq 0, b \geq 0, a< 0\\
1.3)& a+b \geq 0, b < -1, a\geq 0\\
1.4)& a+b < 0, b \geq 0, a< 0\\
1.5)& a+b < 0, b <-1, a\geq  0\\
1.6)& a+b \geq 0, b \geq 0, a< 0.
\end{align*}
To prove \eqref{o U} we consider the eight cases
\begin{align*} 
2.1)& b+a \geq 0, b-a \geq 0, a\geq 0, b\geq 0\\
2.2)& b+a \geq 0, b-a \geq 0, a<0 , b\geq 0\\
2.3)& b+a \geq 0, b-a <0, a\geq 0, b\geq 0\\
2.4)& b+a \geq 0, b-a <0 , a\geq 0, b<-1 \\
2.5)& b+a <0, b-a \geq 0, a< 0, b\geq 0\\
2.6)& b+a <0 , b-a \geq 0, a< 0, b< -1\\ 
2.7)& b+a <0 , b-a <0 , a\geq 0, b< -1\\
 2.8)&  b+a <0 , b-a < 0, a< 0, b< -1.
\end{align*}
As an example we prove the lemma for case 2.3. Then left side of \eqref{o U} becomes 
\[
(\tilde{h}_{b+a} - \tilde{h}_{a-b-2}) \tilde{h}_{c-1} - \tilde{h}_{b}\tilde{h}_{a+c-1}
\]
which is 
\[
\sum_{j=0}^{a+b} \tilde{h}_{c-1-a-b+2j}-\sum_{j=0}^{a-b-2} \tilde{h}_{c-1-a+b+2+2j} - \sum_{j=0}^{b} \tilde{h}_{c-1+a-b+2j}.
\]
Combine the last two sums to yield 
\[
\sum_{j=0}^{a+b} \tilde{h}_{c-1-a-b+2j}-\sum_{j=0}^{a-1} \tilde{h}_{c-1-a+b+2+2j}.
\]
Now using $a\geq 0$, we combine the above two sums to yield 
\[
\sum_{j=0}^{b} \tilde{h}_{c-1-a-b+2j}.
\]
This proves the lemma for case 2.3. The other cases are similar. This completes the proof. 
\end{proof}

\begin{lemma} \label{l U ker}
For all integers $a,b,c_1,c_2$,
\begin{equation}\label{U ker}
U(a,b,c_1) + U(a,c_1+c_2-b, c_2) \in \mathrm{Ker}(L).
\end{equation}
and 
\begin{equation} \label{U ker s}
U(a,b,b) \in \mathrm{Ker}(L)
\end{equation}
\end{lemma} 
\begin{proof}
The left side of \eqref{U ker} is 
\[
\mathrm{sgn}(a)\left(\sum_{j=0}^{|a|-1} \tilde{h}_{c_1-b-|a|+2j}+ \sum_{j=0}^{|a|-1} \tilde{h}_{-c_1+b-|a|+2j}\right)
\]
We re-index the second sum by $j \mapsto |a|-1-j$ and combine the two sums to yield 
\[
\mathrm{sgn}(a)\sum_{j=0}^{|a|-1} \tilde{h}_{c_1-b-|a|+2j}+  \tilde{h}_{-c_1+b+|a|-2j-2}.
\]
Each summand is a generator  (or possibly 2 times the generator $\tilde{h}_{-1}$) of $\mathrm{Ker}(L)$. 

The left side of equation \eqref{U ker s} is 
\[
\mathrm{sgn}(a)\sum_{j=0}^{|a|-1}\tilde{h}_{-|a|+2j}
\]
which is 
\[
\mathrm{sgn}(a) (\tilde{h}_{-1}\mathbf{1}(a \equiv 1 \mod 2)+\sum_{j=0}^{\lfloor (|a|-1)/2 \rfloor}\tilde{h}_{-|a|+2j}+\tilde{h}_{|a|-2j-2})
\]
Again, each summand is a generator of $\mathrm{Ker}(L)$. This completes the proof. 
\end{proof} 

\section{Radius-value trees} \label{rv trees}

\begin{definition} 
Define a radius-value tree of height $n$ to be a rooted plane ternary tree whose vertices are labeled with even integers such that: 

i) the label of the root of a tree $T$ is called the value of the tree, denoted by $\val(T)$. 

ii) a height 0 tree $T$ consists of of one labeled vertex 

iii) a height $(n+1)$ tree $T$ consists of a 4-tuple $(l, T_1, T_2,T_3)$ where $\val(T) =l$ and where the $T_i$ are three radius-value trees of height $n$ such that 
\[
|l - \val(T_1) - \val(T_3)| \leq \tau(\val(T_2)). 
\]
We call $T_2$ the range-determiner subtree of $T$. Define the sign of $T$ $\mathrm{sgn}(T)$ to be 1 if $T$ is height 0, and otherwise 
\[
\mathrm{sgn}(T) = \mathrm{sgn}(\val(T_2)+1)\mathrm{sgn}(T_1)\mathrm{sgn}(T_2)\mathrm{sgn}(T_3).
\]
Let $\mathrm{RV}_n$ denote the set of all radius-value trees of height $n$. Let $\mathrm{RV}_{n; 2} \subset \mathrm{RV}_n$ consist of trees $T$ such that each leaf vertex of $T$ has label 2. Let $\mathrm{RV}_{n; 2}(k) \subset \mathrm{RV}_{n; 2}$ consist of trees $T$ such that $\val(T) = k$. 

Define an involution $\mathrm{Inv}$ on $\mathrm{RV}_n$ by 
\[
\mathrm{Inv}(T) = T
\]
if $T$ is height 0, and for $T = (\val(T_1)+\val(T_3)+ j, T_1,T_2,T_3)$, define  
\[
\mathrm{Inv}(T) =  (\val(\mathrm{Inv}(T_1))+\val(\mathrm{Inv}(T_3))- j, \mathrm{Inv}(T_1),T_2,\mathrm{Inv}(T_3)). 
\]
\end{definition} 

\begin{definition} 

For a radius-value tree $T \in \mrm{RV}_{n;2}$, define the index of $T$
\[
\ind(T) = \val(T) -2^{n+1}.
\]
 If $T$ is height 0, define $\mathrm{rad}(T)$ to be 0. For $T=(j, T_1,T_2,T_3)$ define 
\[
\mathrm{rad}(T) = \rho(\mathrm{rad}(T_1, T_3),\tau(\val(T_2)), \ind(T_1)+\ind(T_3), j)
\]
where for any trees $T_1, T_3$ we define
\[
\mathrm{rad}(T_1, T_3)=\rho(\mathrm{rad}(T_1),\mathrm{rad}(T_3), \ind(T_1), \ind(T_3)).
\]

Define the subset $\mathrm{RV}_{n; 2}^+ \subset \mathrm{RV}_{n; 2}$ by 
\[
\mathrm{RV}_{0; 2}^+  = \mathrm{RV}_{0; 2}
\]
and for $n \geq1$, to consist of trees $T = (j, T_1,T_2,T_3)$ such that 
\[
T_i \in \mathrm{RV}_{n-1; 2}^+, \,\,\,\,\, i =1,2,3
\]
and 
\[
\val(T_2) \geq \max(\mathrm{rad}(T_2) - 2^n,0). 
\]
\end{definition}

\begin{remark}
The involution $\mathrm{Inv}$ is well-defined because the values of the respective range-determiner subtrees of $T$ and $\mathrm{Inv}(T)$ are equal by construction, and because $\mathrm{Inv}$ preserves the height of a tree, which can be proved by induction. 
\end{remark}

\begin{lemma} \label{l inv on 2}
Suppose $T \in \mathrm{RV}_{n;2}$ and 
\[
\val(T) = 2^{n+1}+j.
\]
Then 
\[
\val(\mathrm{Inv}(T)) = 2^{n+1}-j.
\]
\end{lemma}
\begin{proof}
We use induction on $n$. The statement is true for $n=0$ because the only possible $T$ is the height 0 tree labeled $2$. Assume it is true for some $n=N\geq 0$. Then for a  $T \in \mathrm{RV}_{N+1;2}$
with 
\begin{align*}
T &= (2^{N+1}+ j_1 +2^{N+1}+ j_3+j, T_1,T_2, T_3) \\ 
&= (2^{N+2}+ j_1 + j_3+j, T_1,T_2, T_3) 
\end{align*}
where
\[
\val(T_i) = 2^{N+1}+j_i
\]
and by the induction hypothesis
\[
\val(\mathrm{Inv}(T_i)) = 2^{N+1}-j_i.
\]
 Thus
\begin{align*}
\val(\mathrm{Inv}(T)) &= 2^{N+1}- j_1 +2^{N+1}- j_3-j \\ 
&=2^{N+2}- j_1- j_3-j.
\end{align*}
This completes the proof. 
\end{proof}

\section{Combinatorial expressions and positivity for recurrence relation} \label{pos}

We first consider the recurrence relations from Lemma \ref{l mlcre} in the ring $\tilde{\mrm{CH}}_2$ and prove positivity properties for them. Then we apply the homomorphism $L$ and prove that these properties also extend to the images in $\mrm{CH}_2$. 

\begin{definition} 
For integer $n \geq 0$, we recursively define elements $\tilde{\mathrm{err}}(n,0)$ and $\tilde{\mathrm{err}}(n,1) \in \tilde{\mrm{CH}}_2$ by 
\begin{align*}
\tilde{\mathrm{err}}(0,0) &= \tilde{h}_2\\ 
\tilde{\mathrm{err}}(0,1) &= \tilde{h}_1,
\end{align*}
and
\begin{align}
\tilde{\mathrm{err}}(n+1,0) &= (\tilde{\mathrm{err}}(n+1,0)^2 - \tilde{\mathrm{err}}(n+1,1)^2)  \tilde{\mathrm{err}}(n+1,0)  \label{0 rec}\\ 
\tilde{\mathrm{err}}(n+1,1) &= (2\tilde{\mathrm{err}}(n+1,0)^2   -\tilde{h}_1\tilde{\mathrm{err}}(n+1,1)\tilde{\mathrm{err}}(n+1,0))\tilde{\mathrm{err}}(n+1,1) \label{1 rec}
\end{align}
\end{definition} 

\begin{theorem} \label{t E RV}
For each integer $n \geq 0$, define $E(n) \in \tilde{\mrm{CH}}_2$ by  
\begin{align*}
E(n) &= \sum_{T \in \mathrm{RV}_{n;2}} \mathrm{sgn }(T)\tilde{h}_{\val(T)}\\ 
\end{align*}
Then
\begin{align}
\tilde{\mathrm{err}}(n,0) &\equiv E(n) \mod \mathrm{Ker}(L) \label{E 0 equiv}\\ 
\tilde{\mathrm{err}}(n,1) &\equiv S_{-1}(E(n)) \mod \mathrm{Ker}(L). \label{E 1 equiv}
\end{align}
\end{theorem} 
\begin{proof} 
We use induction on $n$. The theorem is true for $n=0$ because $\mathrm{RV}_{n;2}$ consists of the tree with a single vertex labeled $2$. Thus 
\begin{align*}
E(0) = \tilde{h}_2 = \tilde{\mathrm{err}}(0,0) \\ 
S_{-1}(E(0)) = \tilde{h}_1 = \tilde{\mathrm{err}}(0,1).
\end{align*}
Assume the theorem is true for some $n=N \geq 0$. From equation \eqref{0 rec}, we must prove 
\[
E(N+1) \equiv E(N)^3 - S_{-1}(E(N))^2E(N) \mod \mathrm{Ker}(L).
\]
We express the right side as a sum over triples 
\begin{equation} \label{init e  sum}
\sum_{(T_1, T_2, T_3) \in \mathrm{RV}_{N;2}^3}\mathrm{sgn}(T_1, T_2, 3)( \tilde{h}_{\val(T_1)} \tilde{h}_{\val(T_2)}  - \tilde{h}_{\val(T_1)-1}\tilde{h}_{\val(T_2)-1} ) \tilde{h}_{\val(T_3)}
\end{equation}
where we denote
\[
\mathrm{sgn}(T_1, T_2, 3) = \prod_{i=1}^3 \mathrm{sgn}(T_i)
\]
From Lemma \ref{l oe cont}, equation \eqref{e cont}, the sum \eqref{init e sum} is equal to 
\[
\sum_{(T_1, T_2, T_3) \in \mathrm{RV}_{N;2}^3} \mathrm{sgn}(T_1, T_2, 3) \tilde{h}_{\val(T_1)+\val(T_2)}   \tilde{h}_{\val(T_3)},
\]
which by Lemma \ref{l U}, equation \eqref{e U},  is equal to 
\begin{equation} \label{U triple e}
\sum_{(T_1, T_2, T_3) \in \mathrm{RV}_{N;2}^3} \mathrm{sgn}(T_1, T_2, 3) \left(\tilde{h}_{\val(T_2)}   \tilde{h}_{\val(T_1)+\val(T_3)}+ U(\val(T_1),\val(T_2), \val(T_3))\right). 
\end{equation}
Now we define an involution $\mathrm{Inv}_{2,3}$ on $\mathrm{RV}_{N;2}^3$ by 
\[
\mathrm{Inv}_{2,3}((T_1, T_2, T_3) ) = (T_1, \mathrm{Inv}(T_2), \mathrm{Inv}(T_3)). 
\]
Let 
\[
\mathrm{RV}_{N;2}^{(i)}/\mathrm{Inv}_{2,3}
\]
denote the set of equivalence class representatives from those equivalence classes of size $i$ for $\mathrm{RV}_{N;2}$ under the action of $\mathrm{Inv}_{2,3}$, where $i=1$ or 2. 
We express the sum \eqref{U triple e} as 
\begin{align}
&\sum_{(T_1, T_2, T_3) \in \mathrm{RV}_{N;2}} \mathrm{sgn}(T_1, T_2, 3) (\tilde{h}_{\val(T_2)}   \tilde{h}_{\val(T_1)+\val(T_3)} ) \label{hb}\\ 
+ &\sum_{(T_1, T_2, T_3) \in \mathrm{RV}_{N;2}^{(2)}/\mathrm{Inv}_{2,3}} \mathrm{sgn}(T_1, T_2, 3) (U(\val(T_1),\val(T_2), \val(T_3)) + U(\val(T_1),\val(\mathrm{Inv}(T_2)), \val(\mathrm{Inv}(T_3)))) \label{2 class} \\
+&\sum_{(T_1, T_2, T_3) \in \mathrm{RV}_{N;2}^{(1)}/\mathrm{Inv}_{2,3}} \mathrm{sgn}(T_1, T_2, 3) U(\val(T_1),\val(T_2), \val(T_3)). \label{1 class} 
\end{align}

Each summand in sum \eqref{2 class}, by Lemma \ref{l inv on 2} is up to sign of the form 
\[
 U(\val(T_1),2^{N+1}+j_2, 2^{N+1}+j_3) + U(\val(T_1),2^{N+1}-j_2, 2^{N+1}-j_3))
\]
which is in $\mathrm{Ker}(L)$ by applying Lemma \ref{l U ker},  line \eqref{U ker} with 
\begin{align*} 
a &= \val(T_1)\\ 
 b &= 2^{N+1}+j_2 \\ 
 c_1 &= 2^{N+1}+j_3 \\ 
  c_2 &= 2^{N+1}-j_3.
\end{align*}

And each summand in sum \eqref{1 class} is up to sign of the form 
\begin{equation} \label{U 1 term}
U(\val(T_1),2^{N+1}, 2^{N+1})
 \end{equation}
 because a triple $(T_1, T_2, T_3)$ that gets mapped to itself under $\mathrm{Inv}_{2,3}$ must have 
 \[
 \val(\mathrm{Inv}(T_i)) = \val(T_i) = 2^{N+1}
 \]
 for $i=2,3$, where we have used Lemma \ref{l inv on 2}. Expression \eqref{U 1 term} is in $\mathrm{Ker}(L)$ by Lemma \ref{l U ker}, line \eqref{U ker s}. 
 
Thus we have proved that 
\begin{equation} \label{err0 N+1 sum mod}
\err(N+1, 0) \equiv \sum_{(T_1, T_2, T_3) \in \mathrm{RV}_{N;2}} \mathrm{sgn}(T_1, T_2, 3)\tilde{h}_{\val(T_2)}   \tilde{h}_{\val(T_1)+\val(T_3)} \mod \mathrm{Ker}(L).
\end{equation}
 The above sum is equal to 
 \begin{align*}
 &\sum_{(T_1, T_2, T_3) \in \mathrm{RV}_{N;2}} \mathrm{sgn}(T_1, T_2, 3)\mathrm{sgn}(\val(T_2)+1)\sum_{j=0}^{\tau(\val(T_2))}  \tilde{h}_{\val(T_1)+\val(T_3) - \tau(\val(T_2)+2j)}\\ 
  =& \sum_{(l, T_1, T_2, T_3), T_i \mathrm{RV}_{N;2}, |l - \val(T_1)-\val(T_3)| \leq \tau(\val(T_2), l\equiv 0 \mod 2)}\mathrm{sgn}(T_1, T_2, 3)\mathrm{sgn}(\val(T_2)+1)\tilde{h}_l \\ 
  = & \sum_{T \in \mathrm{RV}_{N+1; 2}} \mathrm{sgn}(T) \tilde{h}_{\val(T)}
 \end{align*}
 This completes the induction step for equation \eqref{E 0 equiv}.

Now we prove \eqref{E 1 equiv}. Note for any $g \in \tilde{\mrm{CH}}_2$
\[
\tilde{h}_1g = S_{-1}(g)+ S_1(g)
\]
so 
\begin{align*}
\tilde{h}_1\err(N,1) &= S_{-1}(\err(N,1))+ S_1(\err(N,1)) \\ 
& = S_{-2}(\err(N,0))+ \err(N,0)
\end{align*}
by the induction hypothesis. Thus from equation \eqref{1 rec} we must prove
\[
E(N+1) \equiv E(N)^2 S_{-1}(E(N))-S_{-2}(E(N))E(N)S_{-1}(E(N)) \mod \mathrm{Ker}(L).
\]
We express the right side as a sum over triples 
\begin{equation} \label{init o sum}
\sum_{(T_1, T_2, T_3) \in \mathrm{RV}_{N;2}^3}\mathrm{sgn}(T_1, T_2, 3)( \tilde{h}_{\val(T_1)} \tilde{h}_{\val(T_2)}  - \tilde{h}_{\val(T_1)-2}\tilde{h}_{\val(T_2)} ) \tilde{h}_{\val(T_3)-1}.  
\end{equation}
From Lemma \ref{l oe cont}, equation \eqref{o cont}, the sum \eqref{init o sum} is equal to 
\[
\sum_{(T_1, T_2, T_3) \in \mathrm{RV}_{N;2}^3} \mathrm{sgn}(T_1, T_2, 3) (\tilde{h}_{\val(T_2)+\val(T_1)}+\tilde{h}_{\val(T_2)-\val(T_1)})    \tilde{h}_{\val(T_3)-1},
\]
which by Lemma \ref{l U}, equation \eqref{o U},  is equal to 
\begin{equation} \label{U triple o}
\sum_{(T_1, T_2, T_3) \in \mathrm{RV}_{N;2}^3} \mathrm{sgn}(T_1, T_2, 3) \left(\tilde{h}_{\val(T_2)}   \tilde{h}_{\val(T_1)+\val(T_3)-1}+ U(\val(T_2)+1,\val(T_1), \val(T_3))\right). 
\end{equation}
Now we define an involution $\mathrm{Inv}_{1,3}$ on $\mathrm{RV}_{N;2}^3$ by 
\[
\mathrm{Inv}_{1,3}((T_1, T_2, T_3) ) = ( \mathrm{Inv}(T_1),T_2, \mathrm{Inv}(T_3)). 
\]
Let 
\[
\mathrm{RV}_{N;2}^{(i)}/\mathrm{Inv}_{1,3}
\]
denote the set of equivalence class representatives from those equivalence classes of size $i$ for $\mathrm{RV}_{N;2}$ under the action of $\mathrm{Inv}_{1,3}$, where $i=1$ or 2. 
We express the sum \eqref{U triple o} as 
\begin{align}
&\sum_{(T_1, T_2, T_3) \in \mathrm{RV}_{N;2}} \mathrm{sgn}(T_1, T_2, 3) (\tilde{h}_{\val(T_2)}   \tilde{h}_{\val(T_1)+\val(T_3)-1} ) \label{sum class o}\\ 
+ &\sum_{(T_1, T_2, T_3) \in \mathrm{RV}_{N;2}^{(2)}/\mathrm{Inv}_{1,3}} \mathrm{sgn}(T_1, T_2, 3) (U(\val(T_2)+1,\val(T_1), \val(T_3)) + U(\val(T_2)+1,\val(\mathrm{Inv}(T_1)), \val(\mathrm{Inv}(T_3)))) \label{2 class o} \\
+&\sum_{(T_1, T_2, T_3) \in \mathrm{RV}_{N;2}^{(1)}/\mathrm{Inv}_{1,3}} \mathrm{sgn}(T_1, T_2, 3) U(\val(T_2)+1,\val(T_1), \val(T_3)). \label{1 class o} 
\end{align}
By the same reasoning above for the case $\err(N+1,0)$, the sums \eqref{2 class o} and \eqref{1 class o} are in $\mathrm{Ker}(L)$. The same reasoning obtaining $E(N+1)$ also holds, but now yielding $S_{-1}(E(N+1))$ because of $\tilde{h}_{\val(T_1)+\val(T_3)-1}$ in sum \eqref{sum class o} instead of $\tilde{h}_{\val(T_1)+\val(T_3)}$. This completes the induction step for equation \eqref{E 1 equiv}. This completes the proof.
\end{proof} 

\begin{theorem} \label{t partition}
The set $\mathrm{RV}_{n;2}^+$ may be partitioned into a set $\mathrm{Paths}_n$ of central paths $P$ such that: 

i) $\mathrm{rad}(P) \equiv 0 \mod 2 $

ii) if $P(j) = T$, then $\ind(T) = j$ 

iii) if $T \in V(P)$, then $ \mathrm{rad}(P) =   \mathrm{rad}(T)$. 
\end{theorem} 
\begin{proof} 
We use induction on $n$. The three statements are true for $n=0$ as the set consists of one path with one vertex. Assume they are true for some $N \geq 0$. 
We construct the set of $4$-tuples 
\[
(l, T_1,T_2,T_3)
\]
that constitute $\mathrm{RV}_{n;2}^+$ as the set of pairs 
\[
((l, T_2) , (T_1,T_3))
\]
and show how this product set may be partitioned into a set of centered paths.

By the induction hypothesis, let $\mathrm{Paths}_N$ be the set of centered paths such that 
\[
\mathrm{RV}_{N;2}^+ = \dot \bigcup_{P \in \mathrm{Paths}_N} V(P).
\]

Then 
\begin{align*}
(\mathrm{RV}_{N;2}^+)^2 &= \dot \bigcup_{(P_1,P_2) \in  \mathrm{Paths}_N^2} V(P_1) \times V(P_2) \\ 
\end{align*}
and each
\[
V(P_1) \times V(P_2) = \dot \bigcup_{k=0}^{\min(\mathrm{rad}(P_1), \mathrm{rad}(P_2))} V(\pi(P_1,P_2,k)).
\]
Thus let $ \mathrm{Paths}_N'$ denote 
\[
\mathrm{Paths}_N'=\dot \bigcup_{(P_1,P_2)} \dot \bigcup_{k=0}^{\min(\mathrm{rad}(P_1), \mathrm{rad}(P_2))} \pi(P_1,P_2,k).
\]
Each path $P \in \mathrm{Paths}_N'$ is of even radius because each $P_1$ and $P_2$ has even radius by induction hypothesis. Furthermore, by construction $\rad{P} = \rad(T_1,T_3)$. Suppose 
\begin{align*}
P(j) &= (P_1(x),P_2(y)) \\ 
      & = (T_1,T_3).
\end{align*}
By the induction hypothesis $x = \ind(T_1), y = \ind(T_3)$, and by Lemma \ref{l cp facts}, iv),
\[
j = x+ y
\] 
so 
\[
j = \ind(T_1) + \ind(T_3).
\]
Now for $T_2 \in  \mathrm{RV}_{N;2}^+$, by induction hypothesis $T_2$ is in exactly one path $P \in \mathrm{Paths}_N$. If $T_2 = P(-j)$ with 
\[
\val(T_2) = 2^{N+1}-j <0,
\]
then we must have $\rad(T_2)\geq j$, so we can pair this path vertex with $P(-2^{N+2}+j-2)$, which we denote by $\iota(T_2)$, which satisfies
\[
\val(P(-2^{N+2}+j-2)) = -2^{N+1}+j-2.
\]
In the sum for $E(N+1)$, the contributions from 
\[
(l, T_1,P(-j), T_3) \,\,\,\,\,\,\text{ and } (l, T_1,P(-2^{N+2}+j-2), T_3)
\]
have opposite sign and thus cancel each other. $P(-j)$ corresponds to the element $\tilde{h}_{\val(T_2)}$ in the sum \eqref{hb} with $b = 2^{N+1}-j$, and with 
\[
\tilde{h}_{2^{N+1}-j}\tilde{h}_{a+c} = - \tilde{h}_{-2^{N+1}+j-2} \tilde{h}_{a+c}. 
\]
Omitting these pairs from $\mrm{RV}_{N;2}^+$ results in the set 
\[
\mrm{RV}_{N;2}' = \{T_2  \in \mrm{RV}_{N;2}^+ \colon  \val(T_2) \geq \max(\mathrm{rad}(T_2) - 2^n,0)\}.
\]
Such $T_2$ have 
\[
\mrm{sgn}(\val(T_2)+1) =1
\]
by construction. 

Now for each $T_2 \in \mrm{RV}_{N;2}'$ we construct the centered path $P_{T_2}$
\[
P_{T_2}(-\val(T_2)+2j) \rightarrow P_{T_2}(-\val(T_2)+2j+2) 
\]
for $0 \leq j \leq \val(T_2) - 1$ where the vertex  $P_{T_2}(j) $ is the ordered pair 
\[
(j, T_2).
\] 
B Lemma \ref{l cp facts} iii), we may partition 
\[
V(\pi(P_1, P_2, k))   \times V(P_{T_2})
\]
into a set of centered paths $\mrm{Paths}_{N+1}$. A vertex of $P \in \mrm{Paths}_{N+1}$ is of the form 
\begin{equation} \label{2 pair}
P(l) = ((T_1,T_3), (j, T_2)) 
\end{equation} 
Thus we write \eqref{2 pair} which we write as a 4-tuple 
\[
 (j,T_1,T_2,T_3) \in \mrm{RV}_{N;2}^+,
\]
and taking the union over all such $T_2, P_1, P_2, k$ yields $\mrm{RV}_{N;2}^+$. 
By Lemma \ref{l cp facts}, iv),  the index $l$ of $P(l)$ is 
\[
l = \ind(T_1)+\ind(T_3) + j
\]
and 
\begin{align*}
\ind( (j,T_1,T_2,T_3)) &= \val(T_1) + \val(T_3)+ j -2^{N+2} \\ 
& = 2^{N+1}+ \ind(T_1) + 2^{N+1}+ \ind(T_3)+ j - 2^{N+2} 
&= \ind(T_1)+\ind(T_3) + j. 
\end{align*}
This proves statement $ii$. 
Also by Lemma \ref{l cp facts}, ii) 
\[
\rad(P) = \rho(\rad(T_1,T_3), \val(T_2),\val(T_1)+\val(T_3),j)
\]
which is the definition of 
\[
\rad( (j,T_1,T_2,T_3)).
\]
This proves statement $iii$. 

Each path  $P \in \mrm{Paths}_{N+1}$ has even radius by Lemma \ref{l cp facts}, i, because $P$ is built from the product paths which by the induction hypothesis have even radius. This proves statement $i$.  
This completes the proof. 
\end{proof}

\begin{corollary} \label{c partition}
\begin{align}
E(n) &= \sum_{T \in \mathrm{RV}_{n;2}^+} \tilde{h}_{\val(T)}\label{E sum +}\\
&= \sum_{k=2^{n} - r(n))}^{ 2^n + r(n)}  c_{2k}\tilde{h}_{2k}
\end{align}
where $r(n) = 3^n - 2^n$. The $c_{2k}$ satisfy for $i\geq 0$    

i) $c_{2k} > 0$

ii) $c_{2^{n+1}-2i} = c_{2^{n+1}+2i}$ 

iii)  $c_{2^{n+1}-2i}\geq c_{2^{n+1}-2i-2}$. 

Furthermore 
\begin{equation} \label{sum prime}
L(E(n)) = \sum_{T \in  \mrm{RV}_{N;2}'} h_{\val(T)}.
\end{equation}
\end{corollary}
\begin{proof}
From the proof of Theorem \ref{t partition}, we construct a sign-reversing involution on the set
\[
W_n= \mathrm{RV}_{n;2} -  \mathrm{RV}_{n;2}^+.
\] 
For $T \in W_n$, let 
\[
(l , T_1,T_2,T_3)
\]
be the leftmost subtree of $T$ of height $k$ such that 
\[
(l , T_1,T_2,T_3) \notin \mathrm{RV}_{k;2}^+
\]
with $k$ minimal. Then let $T' \in W_n$ be the tree obtained from $T$ by replacing its subtree $(l , T_1,T_2,T_3)$ with 
\[
(l , T_1,\iota(T_2),T_3).
\]
This is a valid radius-value tree, i.e. $l$ is in the proper range, because 
\[
\tau(\val(\iota(T_2))) = \tau(\val(T_2))
\] 
by construction, and since exactly one of $\{ \val(T_2)+1, \val(\iota(T_2))+1\} $ is negative, we have 
\[
\mrm{sgn}(T_2) = - \mrm{sgn}(\iota(T_2)). 
\]
For $T \in \mathrm{RV}_{n;2}^+$, every subtree $(l , T_1,T_2,T_3)$ of $T$ therefore has 
\[
\val(T_2)+1 \geq 0,
\]
so $\mrm{sgn}(T) =1$. This proves equation \eqref{E sum +}. 

Theorem \ref{t partition}, ii, implies that the central vertex of each centered path in $\mrm{Paths}_n$ has value $2^{n+1}$. The maximum value that a tree in $\mathrm{RV}_{n;2}$ attains is $2(3^n)$ because this maximum at height $n$ is 3 times the maximum at height $n-1$, as the radius-value trees are ternary. Such a tree is in $\mathrm{RV}_{n;2}^+$ because by choosing the maximum values at each subtree $(l , T_1,\iota(T_2),T_3)$, there can be no $T_2$ with 
\[
\val(T_2) +1 <0. 
\]
 Statements $i, ii, iii$ follow immediately from the result that  $\mathrm{RV}_{n;2}^+$ can be partitioned into centered paths, each of which has central vertex which is a tree with value $2^{n+1}$. 
 
 The involution $\iota$ corresponds to pairing 
 \[
 L({\tilde{h}_{-i}} )= -  L({\tilde{h}_{i-2}}). 
 \]
 This proves equation \eqref{sum prime}. This completes the proof. 
\end{proof}

\end{document}